\documentclass[11pt]{amsart}
\usepackage{
	amssymb,
	amsmath}

\usepackage[bookmarks,colorlinks,pagebackref]{hyperref} 

\usepackage{graphicx}

\usepackage{hyperref}

\usepackage[latin1]{inputenc}

\usepackage{mathpazo}
\usepackage[scaled=.95]{helvet}
\usepackage{courier}

\numberwithin{equation}{section}
\theoremstyle{plain} 
\newtheorem{proposition}{Proposition}[section]  
\newtheorem{lemma}[proposition]{Lemma}
\newtheorem{corollary}[proposition]{Corollary} 
\newtheorem{theorem}[proposition]{Theorem} 

\theoremstyle{definition} 
\newtheorem{definition}[proposition]{Definition}

\newtheorem{remark}[proposition]{Remark} 

\newtheorem{example}[proposition]{Example} 

\newtheorem{notation}[proposition]{Notation}

\newtheorem{notation and recalls}[proposition]{Notations and Recalls}

\newcommand\Hom{\operatorname{Hom}}

\newcommand\Ext{\operatorname{Ext}}

\newcommand\im{\operatorname{Im}}

\newcommand{\xx}{\underline x}

\newcommand{\ff}{\underline f}
\newcommand{\UU}{\underline U}

\newcommand\Spec{\operatorname{Spec}}

\newcommand{\qism}{\stackrel{\sim}{\longrightarrow}}

\author[P.~Schenzel]{Peter Schenzel}
\title[Pro-zero homomorphisms]
{On Pro-zero homomorphisms and sequences in local (co-)homology}

\address{Martin-Luther-Universit\"at Halle-Wittenberg,
	Institut f\"ur Informatik, D --- 06 099 Halle (Saale), Germany}
\email{schenzel@informatik.uni-halle.de}

\date{\today}

\begin{document}

	\begin{abstract} 
		Let $\xx$ denote a system of elements of a commutative ring $R$. For an $R$-module 
		$M$ we investigate when $\xx$ is $M$-pro-regular resp. $M$-weakly 
		pro-regular as  generalizations of $M$-regular sequences. This is done  in terms of 
		\v{C}ech co-homology resp. homology, defined by $H^i(\check{C}_{\xx} 
		\otimes_R \cdot)$ resp. by $H_i({\textrm{R}}  \Hom_R(\check{C}_{\xx},\cdot)) \cong H_i(\Hom_R(\mathcal{L}_{\xx},\cdot))$, where $\check{C}_{\xx}$ denotes the 
		\v{C}ech complex and $\mathcal{L}_{\xx}$ is a bounded free resolution of it as constructed in \cite{SS} 
		resp. \cite{Sp11}. The property of $\xx$ being $M$-pro-regular resp. $M$-weakly 
		pro-regular follows by the vanishing of certain  \v{C}ech co-homology resp. 
		homology modules, which is related to completions.  This extends previously work by 
		Greenlees and May (see \cite{GM}) and Lipman et al. (see \cite{ALL1}). This contributes to a further understanding of \v{C}ech (co-)homology in the non-Noetherian case.
		As a technical tool we use one of Emmanouil's 
		results (see \cite{Ei}) about the inverse limits and its derived functor. As an application we prove a global variant 
		of the results with an application to prisms in the sense of Bhatt and Scholze (see \cite{BhS}).
	\end{abstract}


	\maketitle
	
	\section{Introduction}
	Let $R$ denote a commutative ring with $\xx = x_1,\ldots,x_r$ a system of elements. For an $R$-module $M$ 
	we study generalizations of a $M$-regular sequence called $M$-pro-regular sequence and 
	$M$-weakly pro-regular sequence. To this end we denote 
	by $\check{C}_{\xx}$ the \v{C}ech complex with respect to $\xx$ (see e.g. \cite[6.1]{SS}). It is 
	a bounded complex of flat $R$-modules. For an $R$-module $M$ we write $\check{C}_{\xx}(M) 
	= \check{C}_{\xx} \otimes_RM$. We call $\check{H}^i_{\xx}(M) = H^i(\check{C}_{\xx}(M)), i \in 
	\mathbb{Z}$, the \v{C}ech cohomology of $M$. Dually we look at the complex ${\rm{R}} \Hom_R(\check{C}_{\xx},M)$  
	in the derived category. There is a  free resolution of $\check{C}_{\xx}$ by 
	a bounded complex $\mathcal{L}_{\xx}$ and $\Hom_R(\mathcal{L}_{\xx},M)$ is a representative 
	of ${\rm{R}} \Hom_R(\check{C}_{\xx},M)$ (see \cite{SS} and \cite{Sp11}). We define $\check{H}_i^{\xx}(M) = H_i(\Hom_R(\mathcal{L}_{\xx},M)) 
	 \cong H_i( {\rm{R}} \Hom_R(\check{C}_{\xx},M)), i \in \mathbb{Z},$ as the \v{C}ech homology of $M$. 
	For the case of $R$ a Noetherian ring let $\mathfrak{a} = \xx R$ then it follows that 
	$\check{H}^i_{\xx}(M) \cong 
	H^i_{\mathfrak{a}}(M)$, the $i$-th local cohomology of $M$ with support in $\mathfrak{a}$. 
	At first this was established by Grothendieck (see \cite{Ga3} and \cite{Ga2}). 
	Dually, for Noetherian rings $R$ we have $\check{H}^{\xx}_i(M) \cong \Lambda_i^{\mathfrak{a}}(M)$, where $\Lambda_i^{\mathfrak{a}}(\cdot)$ 
	denotes the left derived functors of the completion $\Lambda^{\mathfrak{a}}(\cdot)$. Contributions 
	were done by Matlis (see \cite{Me1}), Simon (see \cite{Sam1}), Greenlees and May (see \cite{GM}) 
	and others. 
	
	Starting with Greenlees and May (see \cite{GM})  and Lipman et al. (see \cite{ALL1}) there were extensions to non-Noetherian rings with  sequences $\xx$ that are called pro-regular resp. weakly pro-regular 
	(see below for the definitions). 
	In particular, when $\xx$ is weakly pro-regular the isomorphisms $\check{H}^i_{\xx}(M) \cong 
	H^i_{\mathfrak{a}}(M)$ and $\check{H}^{\xx}_i(M) \cong \Lambda_i^{\mathfrak{a}}(M)$ hold for any 
	$i \in \mathbb{Z}$ and any $R$-module $M$ and more generally for any complex $X \in D(R)$ 
	(see \cite{PSY}, \cite{Pl} ,   \cite{Sp2} and  \cite{SS} for more details). 
	
	In the situation of $\xx$ an $R$-regular sequence there is a corresponding property of $\xx$ being an $M$-regular sequence (see e.g. \cite{Mh}). This is a challenge for the study of the relative 
	version that $\xx$ is weakly $M$-regular for modules instead of $M=R$. 
	Namely, $\xx$ is called  an $M$-weakly pro-regular sequence 
    (see also \cite[7.3.1]{SS}) provided the inverse system $\{H_i(\xx^{(n)};M)\}_{n \geq 1}$ is
	pro-zero for $i = 1, \ldots,r$, i.e. for each $n$ there is an integer $m \geq n$ such that the natural 
	map $H_i(\xx^{(m)};M) \to H_i(\xx^{(n)};M)$ is zero. 
	Here $\xx^{(n)} = x_1^n,\ldots,x_r^n$ and 
	$H_i(\xx^{(n)};M)$ denotes the Koszul homology.  An $R$-weakly pro-regular sequence is called weakly pro-regular. For a  first description of $M$-weakly 
	pro-regular sequences see \cite[Theorem 4.2]{Sp12}. Let $\widehat{M}^{\xx} = \Lambda^{\xx}(M)$ denote 
	the $\xx$-adic completion of $M$. 
	
	\begin{theorem} \label{in-1}
		For an $R$-module $M$ and a sequence $\xx = x_1,\ldots,x_r$ the following 
		is equivalent:
		\begin{itemize}
			\item[(i)] $\xx$ is $M$-weakly pro-regular.
			\item[(ii)] $\check{C}_{\xx}(\Hom_R(M,I))$ is a right resolution of $\Hom_R(\Lambda^{\xx}(M),I)$ 
			for any injective $R$-module $I$. 
			\item[(iii)] $\Hom_R(\mathcal{L}_{\xx}, M\otimes_R F)$ is a left resolution of $\Lambda^{\xx}(M \otimes_R F)$ for any free $R$-module $F$.
			\item[(iv)] $\Hom_R(\mathcal{L}_{\xx}, X)$ is a left resolution of $\Lambda^{\xx}(X)$ for  $X = M, M[T]$.
			\item[(v)] $\Hom_R(\mathcal{L}_{\xx}, M[T])$ is a left resolution of $\Lambda^{\xx}(M[T])$. 
		\end{itemize}
	\end{theorem} 
	
	Note that the equivalence of (i), (iii) and (iv) in the particular case of $M = R$ was shown by 
	Positselski (see \cite[Theorem 3.6]{Pl}), that is in the case when $\xx$ is $R$-weakly pro-regular (or weakly pro-regular for short).  
	Then the complexes $\Hom_R(\mathcal{L}_{\xx},X)$ and $\textrm{L} \Lambda^{\xx}(X)$ are isomorphic 
	in the derived category for all $X \in D(R)$ (see \cite{PSY} generalizing the case of bounded complexes 
	shown in  \cite{Sp11}).  For the proof of \ref{in-1} and the notion of left/right resolution see the comments after \ref{apa-11}.
	
		The notion of a weakly pro-regular sequence $\xx = x_1,\ldots,x_r$  is defined in terms of the 
		Koszul homology of the whole sequence $\xx$. An $M$-regular sequence is defined by the vanishing of 
		$\xx_{i-1}M:_M x_i/\xx_{i-1}M$ for $i = 1,\ldots,r$, 	where $\xx_{i-1} = x_1,\ldots,x_{i-1}$. 
		As a generalization of that Greenlees and May (see \cite{GM}) 
		resp. Lipman et al. (see \cite{ALL1}) invented the notion of an $M$-pro-regular sequence. 
		Note that  both of the definitions  are equivalent (see \cite[Proposition 2.2]{Sp12}).
	A sequence $\xx$ is called $M$-pro-regular 
	if the inverse system $\{\xx_{i-1}^{(n)}M :_M  x_i^n/\xx_{i-1}^{(n)}M )\}_{n \geq 1}$ with multiplication by 
	$x_i^n$ is pro-zero for $i = 1,\ldots, r$. Note that if $\xx$ is $M$-regular it is also $M$-weakly pro-regular since $\xx_{i-1}^{(n)}M :_M  x_i^n/\xx_{i-1}^{(n)}M  =0$  (see \cite[16.1]{Mh}). 
     A characterization of pro-regular 
	sequences in terms of \v{C}ech cohomology is known (see \cite[Theorem 3.2]{Sp12} and \ref{apa-7}). 
	Here there is a description in the terms of \v{C}ech homology. See  \ref{apa-8} for the following:
	
	\begin{theorem} \label{in-2}
		Let $\xx = x_1,\ldots,x_r$ denote a sequence of elements of $R$. For an $R$-module $M$ the following 
		conditions are equivalent: 
		\begin{itemize}
			\item[(i)] The sequence $\xx$ is $M$-pro-regular. 
			\item[(ii)] $\check{H}^{x_i}_0(\Lambda^{\xx_{i-1}}(M \otimes_RF) )\cong 
			\Lambda^{\xx_i}(M \otimes_RF) $ and $\check{H}^{x_i}_1(\Lambda^{\xx_{i-1}}(M\otimes_RF)) = 0$ 
			for $i = 1,\ldots,r$ and any free $R$-module  $F$.
			\item[(iii)] $\check{H}_0^{\xx_i}(X) \cong \Lambda^{\xx_i}(X)$ and 
			$\check{H}_1^{\xx_i}(X) = 0$ for $i = 1,\ldots,r$ and $X = M,M[T]$.
			\item[(iv)] $\Lambda^{\xx_{i-1}}(M[T])$ is of bounded $x_i$-torsion for $i = 1,\ldots,r$.
		\end{itemize}
	\end{theorem}
	
	 In the final section we apply the previous results to a global situation. To this end 
	we consider a pair $(\mathcal{I},x)$ consisting of an effective Cartier divisor $\mathcal{I} \subseteq R$ 
	and an element $x \in R$ (see \ref{apa-9} for the definitions). We call it pro-regular whenever the
	inverse system $\{H_1(x^n;R/\mathcal{I}^n)\}_{n \geq 1}$ is pro-zero. Then our investigations (see \ref{cor-6}) yield the following:
	
	\begin{corollary} \label{in-3}
		With the previous notation the following conditions are equivalent:
		\begin{itemize}
			\item[(i)] $R/\mathcal{I}$ is of bounded $x$-torsion.
			\item[(ii)]  $(\mathcal{I},x)$ is pro-regular.
			\item[(iii)] $\check{H}_0^x(\Lambda^{\mathcal{I}}(F)) \cong \Lambda^{(x,\mathcal{I})}(F))$ 
			and $\check{H}_1^x(\Lambda^{\mathcal{I}}(F)) = 0$ for any free $R$-module  $F$.
			\item[(iv)] $\Lambda^{\mathcal{I}}(R)$ and $\Lambda^{\mathcal{I}}(R[T])$ are of bounded $x$-torsion. 
		\end{itemize}
	\end{corollary}
	
	As shown in \cite{Sp12} this has applications to prisms in the sense of Bhatt and Scholze (see 
	\cite{BhS}). The equivalent conditions in \ref{in-3} are improvements of the results 
	shown in \cite[Corollary 5.7]{Sp12}. 
	
	In the paper we start with recollections  about inverse limits. In particular we include a different proof of one of 
	Emmanouil's results (see \cite{Ei})
	about  inverse systems needed in the paper. In the third section we prove additional statements 
	about  weakly pro-regular sequences, extending those known before. In section 4 we study pro-regular 
	sequences, continuing the results shown in \cite{Sp12}. Moreover, we prove a necessary and sufficient condition  for the isomorphism  
	$\Lambda^x(\Lambda^{\mathcal{I}}(M)) \cong \Lambda^{(x,\mathcal{I})}(M)$ for an ideal $\mathcal{I} 
	\subset R$ and an element $x \in R$ generalizing a result by Greenlees and May (see 
	\cite[Lemma 1.6]{GM}). 
	Finally in section 5 we study when a
	pair $(\mathcal{I},x)$ consisting of an effective Cartier divisor $\mathcal{I}$ and an element 
	$x \in R$ is pro-regular. Finally we apply these results to prisms in the sense of \cite{BhS} generalizing 
	partial results of \cite{Sp12}. 
	
	In the terminology we follow that of \cite{SS}. In our approach we prefer to work in the category 
	of modules instead of the derived category. For that reason we use a bounded free 
	resolution of the \v{C}ech complex (see \ref{apa-1}).

	\section{Recollections about inverse limits}
	\begin{notation} \label{apn-1}
		(A) Let $R$ denote a commutative ring. Let $\{M_n\}_{n \geq 0}$ be an inverse system of 
		$R$-modules with $\phi_{n,m} : M_m \to M_n$ for all $m \geq n$. Then there is an exact sequence 
		\[
		0 \to \varprojlim M_n \to \prod_{n \geq 0} M_n\stackrel{\Phi}{\longrightarrow} \prod_{n \geq 0} M_n
		\to \varprojlim{}^1 M_n \to 0, 
		\]
		where $\Phi$ denotes the transition map and $ \varprojlim{}^1 M_n$ is the first left derived functor of  the inverse limit  (see e.g. \cite[3.5]{Wc} or \cite[1.2.2]{SS}). \\
		(B) Let $M$ denote an $R$-module. Let $T$ be a variable over $R$. In the following we use $M[|T|]$,  
		the formal power series $R$-module 
		over $M$. That is, the $R$-module $M[|T|]$ consists of all formal series $\sum_{i \geq 0} x_i T^i$ with $x_i \in M$ 
		for all $i \geq 0$. Correspondingly, the $R$-module $M[T]$ consists of all polynomials over $M$. Therefore, 
		$\sum_{i \geq 0} x_i T^i \in M[T]$ if only finitely many $x_i$ are non-zero.  Whence there is an 
		injection $0 \to M[T] \to M[|T|]$ of $R$-modules. \\
		(C) The inverse system $\{M_n\}_{n \geq 0}$  is called pro-zero if for each $n$ there is an integer $m \geq n$ 
		such that the homomorphism $\phi_{n,m}: M_m \to M_n$ is zero. If $\{M_n\}_{n \geq 0}$ is pro-zero, 
		then it is well known that $\varprojlim M_n =  \varprojlim^1 M_n = 0$ since $\Phi$ is an isomorphism (see e.g. \cite[1.2.4]{SS}). \\
		(D) Let $\{M_n\}_{n \geq 0}$  be an inverse system. Then clearly $\im \phi_{n,m'} \subseteq \im\phi_{n,m} 
		\subseteq M_n$ for all $m' \geq m\geq n$. We say that  $\{M_n\}_{n \geq 0}$  
		satisfies the \textit{Mittag-Leffler condition} if for each $n$ the sequence of submodules 
		$\{\im \phi_{n,m} | m \geq n\}$ stabilizes. For instance, this holds if the maps $\phi_{n,m}$ 
		are surjective or $\{M_n\}_{n\geq 0}$ is an inverse system of Artinian $R$-modules. It is well-known that 
		$\varprojlim^1 M_n = 0$ if  $\{M_n\}_{n \geq 0}$ satisfies 
		the Mittag-Leffler condition (see e.g. \cite[1.2.3]{SS}). 
	\end{notation}
	
	For more details about inverse systems we refer to Jensen's exposition in \cite{Jcu} and to \cite{Ei}. 
	It is remarkable that the vanishing in \ref{apn-1} (C) does not imply that 
	$\{M_n\}_{n \geq 0}$ is pro-zero.  
	To this end see the example \cite[1.2.5]{SS} or the following  generalization: 
	
	\begin{example} \label{apn-2}
		Let $(R,\mathfrak{m})$ denote a complete local Noetherian ring 
		with $x \in R$ a non-unit. We consider the direct system $\{R_n\}_{n \geq 0}$ with $R_n = R$ 
		and $\psi_{n,n+1}: R_n \to R_{n+1}$ the multiplication by $x$. Then $\varinjlim R_n \cong R_x$ and 
		there is a short exact sequence 
		$$
		0 \to \oplus_{n \geq 0} R_n \to \oplus_{n \geq 0}R_n  \to R_x  \to 0.
		$$
		Now we apply $\Hom_R(\cdot,R)$ and obtain the inverse system $\{M_n\}_{n \geq 0}$ with 
		$M_n = \Hom_R(R_n,R)$ and 
		with the multiplication $M_{n+1} \stackrel{x}{\to} M_n$. By applying $\Hom_R(\cdot,R)$ 
		to the previous short exact sequence it yields the exact sequence
		\[
		0 \to \Hom_R(R_x,R) \to \prod_{n \geq 0} M_n \to  \prod_{n \geq 0} M_n \to 
		\Ext_R^1(R_x,R) \to 0.
		\]
		Since $R$ is also $xR$-complete 
		$\varprojlim M_n = \Hom_R(R_x,R) = 0$ and $\varprojlim^1 M_n = \Ext_R^1(R_x,R) = 0$ 
		(see \cite[3.1.10]{SS})
		while the inverse system $\{M_n\}_{n \geq 0}$ is neither pro-zero nor satisfies the Mittag-Leffler condition. 
	\end{example}
	
	In the following we shall discuss necessary and  sufficient conditions for an inverse system 
	to be pro-zero.  This extends known results. We need a technical construction.
	
	\begin{remark} \label{apn-4}
		An $R$-module $M$ induces a short exact sequence 
		\[
		0 \to M[T] \stackrel{T}{\longrightarrow} M[T] \to M \to 0,
		\]
		where $T$ denote the shift operator defined by 
		$\sum_{n \geq 0}^{k} x_n T^n \mapsto \sum_{n \geq 0}^{k} x_nT^{n+1}$.
		The inverse system $\{M_n\}_{n \geq 0}$ induces a short exact sequence of inverse systems 
		\[
		0 \to \{M_n[T]\}_{n \geq 0} \stackrel{T}{\longrightarrow} \{M_n[T]\}_{n \geq 0} \to 
		\{M_n\}_{n \geq 0} \to 0,
		\]
		induced by the shift operator.  Then we have the six-term long exact sequence associated 
		to the inverse limit 
		\[
		0 \to  \varprojlim M_n[T] \to  \varprojlim M_n[T] \to  \varprojlim M_n \to 
		\varprojlim{}^1M_n[T] \to  \varprojlim{}^1 M_n[T] \to  \varprojlim{}^1 M_n \to 0
		\]
		(see e.g. \cite[1.2.2]{SS}) . 	
	\end{remark}

	By the Example \ref{apn-2} it follows that the vanishing of $\varprojlim^1 M_n$ is 
	necessary but not sufficient for the Mittag-Leffler condition of the inverse system $\{M_n\}_{n \geq 0}$. 
	A characterization of the Mittag-Leffler condition was shown by Emmanouil (see \cite{Ei}). 
	For our purposes we recall  part of Emmanouil's result (see \cite[Corollary 6]{Ei}). 
	In our argument we use a certain  exact sequence (see the proof of \ref{apn-5})
	and modify an idea  of \cite[tag 0CQA]{stacks} as new ingredients.
	
	\begin{lemma} \label{apn-5}
		Let $\{M_n\}_{n \geq 0}$ denote an inverse system of $R$-modules. Then the 
		following conditions are equivalent: 
		\begin{itemize}
			\item[(i)] $\{M_n\}_{n \geq 0}$ satisfies the Mittag-Leffler condition. 
			\item[(ii)]  $\{M_n[T]\}_{n \geq 0}$ satisfies the Mittag-Leffler condition. 
			\item[(iii)] $\varprojlim^1 M_n = 0$  and $\varprojlim^1 M_n[T] = 0$.
			\item[(iv)] $\varprojlim^1 M_n[T] = 0$.
		\end{itemize}
	\end{lemma}
	
	\begin{proof}
		(i) $\Longrightarrow$ (ii): This follows since the inverse system $\{M_n[T]\}_{n \geq 0} $ satisfies 
		the Mittag-Leffler condition too. \\
		(ii) $\Longrightarrow$ (iv): This holds trivially. \\
		(iii) $\Longleftrightarrow$ (iv): This is a consequence of the six-term exact sequence in \ref{apn-4}.\\
		(iii) $\Longrightarrow$ (i): 
		The injections $0 \to M_n[T] \to M_n[|T|]$ induce a short exact 
				sequence of inverse systems 
				\[
				0 \to  \{M_n[T]\}_{n \geq 0} \to \{M_n[|T|]\}_{n \geq 0}  \to  \{M_n[|T|]/ M_n[T]\}_{n \geq 0}  \to 0.
				\]
				By passing to the inverse limit it provides an exact sequence 
				\[
				0 \to \varprojlim M_n[T] \to  \varprojlim M_n[|T|]  \to  \varprojlim M_n[|T|]/M_n[T]  \to  
				\varprojlim{}^1M_n[T].  
				\]
		Now suppose that $\{M_n\}_{n \geq 0}$ does not satisfy the Mittag-Leffler condition. 
		Then there is an integer $m$ such that the sequence 
		of submodules $\{\im \phi_{m,k} | k \geq m\}$ of $M_m$ does not stabilize. Whence 
		there is an infinite  sequence $m = m_0<  m_1 < \ldots <m_i < \ldots $ 
		and elements $x_i \in M_{m_i}$ such that $ \phi_{m,m_i}(x_i) \in M_m \setminus \phi_{m,m_i+1}(M_{m_i+1})$. 
		Now we define $F = (f_n)_{n \geq 0}  \in \prod_{n \geq 0} M_n[|T|]$ with $f_n = \sum_{i \geq n}  z_{n,i} T^i$ where 
		we put 
		\[
		z_{n,i} = \begin{cases}
			\phi_{n,m_i}(x_i) & if \; m_i \geq n \\
			0 & else. 
		\end{cases}
		\]
		As easily seen $f_n -\phi_{n,n+1}(f_{n+1}) \in M_n[T]$ and $F$ defines an element $F' \in 
		\varprojlim M_n[|T|]/M_n[T]$. Suppose $F'$ has a preimage $G = (g_n)_{n \geq 0} \in \varprojlim M_n[|T|]$ 
		with $g_n = \sum_{i \geq 0} y_{n,i}T^i$ and  $y_{n,i} \in M_n$ for all $i \geq 0$. We have that $y_{n,i} 
		= \phi_{n,n+k} (y_{n+k,i})$ for all $k,i \geq 0$ and therefore $y_{n,i} \in \phi_{n,n+k}(M_{n+k})$.  That is, 
		$y_{m,i} \in \phi_{m,m_i+1}(M_{m_i+1})$ and $y_{m,i} \not= \phi_{m,m_i}(x_i) $ since 
		$ \phi_{m,m_i}(x_i) \in M_m \setminus \phi_{m,m_i+1}(M_{m_i+1})$. Therefore 
		\[
		f_m -g_m = \sum_{i \geq 0}(\phi_{m,m_i}(x_i)  - y_{m,i}) T^i \not\in M_m[T]    
		\]
		and $G$ can not be a preimage of $F'$, a contradiction to the vanishing of $\varprojlim{}^1M_n[T]$.
	\end{proof}

	As a consequence of \ref{apn-5} a characterization of pro-zero inverse systems follows. The 
	vanishing $\varprojlim M_n = \varprojlim{}^1M_n =0$ is not sufficient for $\{M_n\}_{n \geq 1}$ being pro-zero 
	(see \ref{apn-2}). As shown next it follows by the vanishing $\varprojlim M_n[T] = \varprojlim{}^1M_n[T]=0$ (see \ref{apn-3}). For the proof we modify Weibel's argument (see the proof \cite[3.5.7]{Wc}).
	For an $R$-module $M$ and a set $S$ we define 
	$M^{(S)} = \oplus_{s\in S} M_s$ with $M_s = M$. Then it is clear that conditions (iii) and (iv) 
	hold also for the inverse system $\{(M_n)^{(S)}\}_{n \geq 0}$ when they hold for $\{M_n\}_{n \geq 0}$. 
		 
	\begin{corollary} \label{apn-3}
		Let $\{M_n\}_{n \geq 0}$ denote an inverse system of $R$-modules. Then the 
		following conditions are equivalent:
		\begin{itemize}
			\item[(i)] $\{M_n\}_{n \geq 0}$ is pro-zero. 
			\item[(ii)] $\{M_n[T]\}_{n \geq 0}$ is pro-zero. 
			\item[(iii)] $\varprojlim M_n =  \varprojlim^1 M_n = 0$  and $\varprojlim M_n[T]=  \varprojlim^1 M_n[T] = 0$. 
			\item[(iv)]$\varprojlim M_n[T]=  \varprojlim^1 M_n[T] = 0$. 
		\end{itemize}
	\end{corollary}
	
	\begin{proof}
		(i) $\Longrightarrow$ (ii): Because $\{M_n\}_{n \geq 0}$  is pro-zero this holds also for the 
		induced inverse system $\{M_n[T]\}_{n \geq 0}$ as easily seen.\\
		(ii) $\Longrightarrow$ (iv): This is obviously true because  $\{M_n[T]\}_{n \geq 0}$ is pro-zero.
		\\
		(iii) $\Longleftrightarrow$ (iv): This is a consequence of the six-term exact sequence in \ref{apn-4}.\\
			(iii) $\Longrightarrow$ (i):  By view of \ref{apn-5} the inverse system $\{M_n\}_{n \geq 1}$ 
			satisfies the Mittag-Leffler condition. We define $N_n = \im \phi_{n,m}$ where $m = m(n)$ is choosen such that $\{ \im \phi_{n,k}\}_{k \geq n}$ becomes stable. Then $\{N_n\}_{n \geq 1}$ becomes  an inverse system with surjective maps. Because  the inverse system $\{M_n/N_n\}_{n \geq 1}$ is pro-zero  
			the exact sequence $0 \to N_n \to M_n \to M_n/N_n \to 0$ implies 
			$\varprojlim N_n = \varprojlim M_n =0$ and therefore $N_n = 0$. 
	\end{proof}

\section{Weakly pro-regular sequences}
	We start with a few recalls of results and definitions of \cite{SS} and \cite{Sp11}. As above $R$ denotes a 
	commutative ring. 
	
\begin{notation} \label{apa-1}
		(A) For a system of elements $\xx = x_1,\ldots,x_r$ of $R$ let $\check{C}_{\xx}$ denote the \v{C}ech complex 
			\[
		\check{C}_{\xx} :=  \check{C}_{x_1} \otimes_R \cdots \otimes_R \check{C}_{x_r},
		\]
		where $ \check{C}_{x_i} : 0 \to R \to R_{x_i} \to 0$ 
		(see e.g. \cite{Ga2} or  \cite[6.1]{SS}). In the following we look at the complex ${\rm{R}} \Hom_R(\check{C}_{\xx},M)$ 
		for an $R$-module $M$ in the derived category. By virtue of  \cite{GM} there  is a finite free resolution of $\check{C}_{\xx}$.  We follow here the one $\mathcal{L}_{\xx}$ as given in \cite{Sp11}. Whence 
		$\Hom_R(\mathcal{L}_{\xx},M)$ is a representative 
		of ${\rm{R}} \Hom_R(\check{C}_{\xx},M)$. Define the \v{C}ech homology 
		$\check{H}^{\xx}_i(M) = H^{-i}(\Hom_R(\mathcal{L}_{\xx},M))$ and the \v{C}ech cohomology $\check{H}_{\xx}^i(M) = 
		H^i(\mathcal{L}_{\xx} \otimes_RM)$ for all $i \in \mathbb{Z}$ (see \cite{SS} and \cite{Sp11} for more details). \\
		(B) Let $\UU =U_1,\ldots,U_r$ denote a sequence of $r$ variables over $R$. For an 
		$R$-module $M$  we denote, as above, by $M[|\UU|]$ the module of formal power series in the 
		variables $\UU$. Clearly $M[|\UU|] = \varprojlim M[\UU]/\UU^{(n)} M[\UU]$, where 
		$\UU^{(n)} = U_1^n, \ldots,U_r^n$ and $M[\UU]$ is the polynomial module over $M$. 
		For the sequence  $\xx = x_1,\ldots,x_r$ we define the sequence $\xx -\UU = x_1-U_1,\ldots,x_r -U_r$. 
		As one of the main results of the paper \cite[Section 8]{SS} the following isomorphisms are shown 
		\[
		\Hom_R(\mathcal{L}_{\xx},M) \cong K_{\bullet}(\xx-\UU; M[|\UU|]) \cong 
		\varprojlim K_{\bullet}(\xx-\UU; M[\UU]/\UU^{(n)} M[\UU]), 
		\]
		where $ K_{\bullet}(\xx-\UU;\cdot)$ denotes the Koszul complex with respect to the sequence $\xx-\UU$. Moreover there are isomorphisms 
		\[
		\mathcal{L}_{\xx} \otimes_RM \cong K^{\bullet}(\xx-\UU; M[\UU^{-1}]) \cong 
		\varinjlim K^{\bullet}(\xx-\UU; M[\UU]/\UU^{(n)} M[\UU]),
		\]
		where $M[\UU^{-1}]$ denotes the module of inverse polynomials and $K^{\bullet}(\xx-\UU; \cdot)$ 
		is the Koszul co-complex (see  \cite[4.1]{Sp11} for all of the details).
	\end{notation}
	
	In the following there is  technical result for the computation of 
	$ \check{H}^{\xx}_i(M)$ and $ \check{H}_{\xx}^i(M)$ resp. 
	
	\begin{lemma} \label{apa-2}
		We fix the notation of \ref{apa-1}. Furthermore let $\xx^{(n)} = x_1^n,\ldots,x_r^n$ and let  
		$H_i(\xx^{(n)};M)$ denote the Koszul homology and $H^i(\xx^{(n)}; M)$ the Koszul cohomology.
		\begin{itemize}
			\item[(a)]  There are isomorphisms $\check{H}^i_{\xx}(M) \cong \varinjlim H^i(\xx^{(n)};M)$ 
			and short exact sequences 
			\[
			0 \to \varprojlim{}^1 H_{i+1}(\xx^{(n)};M) \to \check{H}^{\xx}_i(M)  \to \varprojlim H_i(\xx^{(n)};M) \to 0,
			\]
			for all $i\in \mathbb{Z}$.
			\item[(b)] For $i > 0$ we have $ \check{H}^{\xx}_i(M) = 0$ if and only if
			$ \varprojlim{}^1 H_{i+1}(\xx^{(n)};M) = \varprojlim H_i(\xx^{(n)};M) = 0$ and 
			$\check{H}^{\xx}_0(M) \cong \Lambda^{\xx}(M)$ if and only if $ \varprojlim{}^1 H_1(\xx^{(n)};M) =0$. 
		\end{itemize}
	\end{lemma}
	
	\begin{proof}
		For the proof of (a) we refer to \cite[6.1.4, 8.1.7]{SS} or \cite[5.6]{Sp11}.  Then (b) is a consequence of the 
		exact sequences in (a).
	\end{proof}
	
	Next we shall give a further characterization for an element $x \in R$ such that an $R$-module $M$ 
	is of bounded $x$-torsion. 
	
	\begin{definition} \label{apa-3}
		(A) Let $M$ denote an $R$-module and $x \in R$ an element. Then $M$ is called of bounded 
		$x$-torsion if the family of increasing submodules $\{0:_M x^n\}_{n \geq 0}$ stabilizes, that is 
		$$
		0:_M x^n = 0:_M x^{n+1} \textrm{ for all } n \gg 0.
		$$ 
		Note that this is equivalent to the fact 
		that the inverse system $\{0:_M x^n\}_{n \geq 0}$ with the multiplication map $0:_M x^m 
		\stackrel{x^{m-n}}{\longrightarrow} 0 :_M x^n, m \geq n,$ being pro-zero. \\
		(B) It is obvious that $M$ is of bounded $x$-torsion if and only if the inverse system of 
		Koszul homology modules $\{H_1(x^n;M)\}_{n \geq 0}$ with the multiplication map 
		$H_1(x^m;M) \stackrel{x^{m-n}}{\longrightarrow} H_1(x^n;M)$ is pro-zero. With this in mind 
		Lipman (see \cite{Lj1}) introduced the generalization of a weakly pro-regular sequence for a 
		ring $R$. For a  generalization 
		to an $R$-module $M$ see \cite[7.3.1]{SS}. That is, a sequence $\xx = x_1,\ldots,x_r$ 
		is called $M$-weakly pro-regular, if for $i > 0$ the inverse system $\{H_i(\xx^{(n)};M)\}_{n \geq 0}$ 
		is pro-zero, where $H_i(\xx^m;M) \to H_i(\xx^n;M), m \geq n,$ denotes the natural map induced by the 
		Koszul complexes.  A first systematic study of $R$-weakly pro-regular sequences has been done in 
		\cite{Sp2}. 
	\end{definition}
	
	For a characterization of $M$-weakly pro-regular sequences see \cite{Sp11}. In fact, this is an extension of 
	$R$-weakly pro-regular sequences shown in \cite{PSY} which extended the results of \cite{Sp2} to unbounded complexes.  Here we shall prove another characterization of 
	$M$-weakly pro-regular sequences. It is  a slight  extension of Potsitselski's result 
	see \cite[Section 3]{Pl}) to the case of an $R$-module $M$. As above, for an $R$-module $M$ and a 
	set $S$ we define $M^{(S)} = \oplus_{s\in S} M_s$ 
	with $M_s = M$.  Note that $M[T] \cong M^{(\mathbb{N})}$. 
	Moreover, $\Lambda^{\xx}(M)  = \widehat{M}^{\xx} = \varprojlim M/\xx^{(n)}M$ 
	denotes the $\xx R$-adic completion of an $R$-module $M$.

	\begin{theorem} \label{apa-4}
		Let $\xx = x_1,\ldots,x_r$ denote a sequence of elements of $R$. For an $R$-module $M$ the following 
		conditions are equivalent:
		\begin{itemize}
			\item[(i)] $\xx$ is $M$-weakly pro-regular.
			\item[(ii)] For any set $S$ it holds $\check{H}^{\xx}_i(M^{(S)}) = 0$ for all $i > 0$ and 
			$\check{H}^{\xx}_0(M^{(S)}) = \Lambda^{\xx}(M^{(S)})$. 
			\item[(iii)]  $\check{H}^{\xx}_i(M[T]) = \check{H}^{\xx}_i(M) = 0$ for all $i > 0$ and 
			$\check{H}^{\xx}_0(M[T]) = \Lambda^{\xx}(M[T])$ and $\check{H}^{\xx}_0(M) = \Lambda^{\xx}(M)$.
			\item[(iv)]$\check{H}^{\xx}_i(M[T]) = 0$ for all $i > 0$ and 
			$\check{H}^{\xx}_0(M[T]) = \Lambda^{\xx}(M[T])$. 
		\end{itemize}
	\end{theorem}
	
	\begin{proof}
		(i) $\Longrightarrow$ (ii): It is clear that for  $i > 0$ the inverse system 
		$\{H_i(\xx^{(n)};M^{(S)})\}_{n \geq 0}$ is pro-zero too. Then 
		$\varprojlim H_i(\xx^{(n)};M^{(S)}) = \varprojlim{}^1 H_i(\xx^{(n)};M^{(S)}) = 0$ 
		for $i > 0$ and (ii) is a consequence of \ref{apa-2}.\\
		(ii) $\Longrightarrow$ (iii)  $\Longrightarrow$ (iv): These hold obviously. \\
		(iv) $\Longrightarrow$ (i): By view of \ref{apa-2} the assumptions imply that 
		\[
		\varprojlim H_i(\xx^{(n)};M[T])=  \varprojlim{}^1  H_i(\xx^{(n)};M[T])= 0 \; \mbox{ for }\;  i > 0.
		\]
		By \ref{apn-3} this completes the proof because  of $H_i(\xx^{(n)};M[T]) \cong 
		H_i(\xx^{(n)};M)[T]$. 
	\end{proof}
	
	In the following example we show that it is not sufficient to assume $S$ to be finite 
	in \ref{apa-4} for the characterization of weakly pro-regular sequences (see also \cite[Example 3.3]{Sp12}).
	
	\begin{example} \label{apa-5}
		Let $R = \Bbbk [|x|]$ denote the formal power series ring in the variable $x$ over the field 
		$\Bbbk$. Then define $A = \prod_{n \geq 1} R/x^nR$. By the component wise  operations $A$ becomes 
		a commutative ring. The natural map $R \to A, r \to (r+ x^nR)_{n \geq 1}$, is a ring homomorphism with 
		$x \mapsto {\bf x} := (x+x^n R)_{n \geq 1}$.  As a direct product of $xR$-complete modules 
		$A$ is an $xR$-complete $R$-module (see \cite[2.2.7]{SS}). 
		Since $R$ is a Noetherian ring $x$ is $R$-weakly pro-regular and 
		$\check{H}^x_i(A) \cong H_i(\Hom_R(\mathcal{L}_x,A))= 0$ for $i > 0$ and  
		$\check{H}^x_0(A) \cong H_0(\Hom_R(\mathcal{L}_x,A)) \cong A$. 
		Moreover, by the change of rings there is an isomorphism 
		$\Hom_R(\mathcal{L}_x, A) \cong \Hom_A(\mathcal{L}_{{\bf x}}, A)$. That is, 
		$\check{H}_i^{{\bf x}} (A) = 0$ for $i > 0$ and $\check{H}_0^{{\bf x}}(A)\cong A$.  
		Now note that  $A$ is not of bounded ${\bf x}$-torsion as easily seen. It follows that the equivalent 
		conditions in \ref{apa-4} do not hold for $A$ and $A[T]$.   
		To be more precise,  recall $H_1(x^n;A) = \prod_{i \geq 1} (x^{i-n}R/x^i R)$ with $x^{i-n} R = R$ 
		for $i \leq n$,  that is
		\[
		H_1(x^n;A) = (\underbrace{R/xR, \ldots, R/x^nR,}_{i \leq n}
		\underbrace{xR/x^{n+1}, \ldots,x^{i-n}R/x^iR, \ldots}_{i > n} ).
		\]
		Therefore $H_1(x^m;A)$ does not stabilize under the multiplication by $x^{m-n}$ in $H_1(x^n;A)$. 
		Note that the $i$-component of the image of $	H_1(x^m;A)$ under the multiplication by $x^{m-n}$ 
		in $H_1(x^n;A)$ is 
		zero for $i \leq m-n < m$ and non-zero for $i = m-n+1$. 
		Whence  
		$\{H_1(x^n;A)\}_{n \geq 1}$  does not satisfy the Mittag-Leffler condition. By view of \ref{apn-5} we have $\varprojlim{}^1 H_1(x^n;A[T]) \not= 0$ and $\Lambda_0^{x}(A[T]) \cong \check{H}^x_0(A[T]) 
		\twoheadrightarrow \Lambda^x(A[T])$ is not an isomorphism (see \ref{apa-2} (a)).
	\end{example}
	
	As an application we have another characterization that an $R$-module $M$ is of 
	bounded $x$-torsion for an element $x \in R$. Note that (iii) in \ref{apa-11} is the analogue to \ref{apa-4} (iv). 
	
	\begin{corollary} \label{apa-11}
		For an element $x \in R$ and an $R$-module $M$ the following conditions are equivalent:
		\begin{itemize}
			\item[(i)] $M$ is of bounded $x$-torsion.
			\item[(ii)] $\check{H}^x_1(M[T]) = \check{H}^x_1(M)= 0$ and $\check{H}^x_0(M[T]) \cong \Lambda^x(M[T])$ 
			and $\check{H}^x_0(M) \cong \Lambda^x(M)$.
			\item[(iii)] $\varprojlim 0 :_{M[T]} x^n=  \varprojlim^1 0 :_{M[T]} x^n = 0$. 
		\end{itemize}
	\end{corollary}
	
	\begin{proof}
		The equivalence of the first two conditions is a particular case of \ref{apa-4}. The equivalence 
		of the first and third condition is a particular case of \ref{apn-3}.
	\end{proof}
	
	Moreover,   the proof of Theorem  \ref{in-1} follows by \ref{apa-4} and \cite[Proposition 5.3]{Sp11}. To this end  note that 
	$\check{H}_i^{\xx}(M) = H_i(\Hom_R(\mathcal{L}_{\xx},M))$. 
	For an $R$-module $X$ we call a complex $X_{\cdot}: \ldots \to  X_1 \to X_0 \to 0$ a left resolution whenever 
	$X_{\cdot} \qism M$. A co-complex $Y^{\cdot}: 0 \to Y^0 \to Y^1 \to \ldots$ is called a right resolution of $X$ 
	provided $X \qism Y^{\cdot}$.
	
	With the previous results we have the following slight generalization of Potsitselski's result (see \cite[Theorem 3.6]{Pl}). Note that  $\xx$ is $R$-weakly pro-regular if it is $R[T]$-weakly pro-regular as easily seen. 
	
	\begin{corollary} \label{in-4}
		For a sequence $\xx = x_1,\ldots,x_r$ of a ring $R$ the following conditions are equivalent:
		\begin{itemize}
			\item[(i)] $\xx$ is $R$-weakly pro-regular.	
			\item[(ii)]  $\Hom_R(\mathcal{L}_{\xx}, M)$ is a left resolution of $\Lambda^{\xx}(M)$ for any free $R$-module $M$.
			\item[(iii)]   $\Hom_R(\mathcal{L}_{\xx}, R[T])$ is a left resolution of $\Lambda^{\xx}(R[T])$.
		\end{itemize}
	\end{corollary}
	
\begin{remark} \label{com}
	While the property of $R$-regular and $M$-regular sequences are quite "symmetric" this is not 
	the case for the notion of weakly pro-regularity. Let $\xx$ denote a sequence of elements of $R$. 
		If it is $R$-weakly pro-regular it follows that $\check{H}^{\xx}_0(M) \cong \Lambda_0^{\xx}(M)$ for 
		any $R$-module $M$  (see e.g. \cite[Chapter 7]{SS}). Let $\xx$ be $M$-weakly pro-regular, then $\check{H}^{\xx}_0(M) \cong \Lambda^{\xx}(M)$ as shown in \ref{apa-4}. Note that the homomorphism 
		$\Lambda_0^{\xx}(M) \to  \Lambda^{\xx}(M)$  is onto (see \cite[2.5.1]{SS}) but in general not an isomorphism 
		(see e.g. Example \ref{apa-5}).
\end{remark}

	\section{Pro-regular sequences}
	
	Before we shall investigate pro-regular sequences we need  technical results about 
	pro-zero inverse systems. To this end let $M$ denote an $R$-module with 
	$\{M_n\}_{n \geq 1}$ a decreasing sequence of submodules of $M$, i.e. $M_{n+1} \subseteq M_n$ 
	for $n \geq 1$. Then $\mathcal{M} = \{M/M_n\}_{n \geq 1}$ forms an inverse system with surjective 
	maps $M/M_{n+1} \to M/M_n$. Moreover, let $\Lambda(\mathcal{M})= \varprojlim M/M_n$. 
	For a sequence of elements 
	$\xx = x_1,\ldots, x_r \in R$ we consider the induced filtration $\{(\xx^{(n)}M,M_n)\}_{n \geq 1}$, where 
	$\xx^{(n)} = x_1^n,\ldots,x_r^n$. We write $\Lambda(\mathcal{M}/\xx \mathcal{M})
	:= \varprojlim M/(\xx^{(n)}M,M_n)$ for the inverse limit of the induced filtration. Then there is a natural homomorphism 
	$\Lambda^{\xx}(\Lambda(\mathcal{M})) \to \Lambda(\mathcal{M}/\xx \mathcal{M})$. In the following we will discuss when it is an isomorphism.
	
	\begin{lemma} \label{gm+}
		With the previous notation there is a short exact sequence 
		\[
		0 \to \varprojlim{}_n \varprojlim{}^1_m H_1(\xx^{(n)};M/M_m) 
		\to \Lambda^{\xx}(\Lambda(\mathcal{M})) \to \Lambda(\mathcal{M}/\xx \mathcal{M}) \to 0.
		\]
		Therefore $ \Lambda^{\xx}(\Lambda(\mathcal{M})) \cong  \Lambda(\mathcal{M}/\xx \mathcal{M})$ 
		if and only if $\varprojlim_n \varprojlim{}^1_m H_1(\xx^{(n)};M/M_m) = 0$.
	\end{lemma}
	
	\begin{proof}
		Let $m,n$ denote positive integers. We investigate the inverse system of Koszul complexes 
		$\{K_{\bullet}(\xx^{(n)};M/M_m)\}_{m \geq 1}$. For its inverse limit there are isomorphisms 
		\[
		\varprojlim{}_m K_{\bullet}(\xx^{(n)},M/ M_m) 
		\cong \Hom_R(K^{\bullet}(\xx^{(n)}),\Lambda(\mathcal{M})) \cong K_{\bullet}(\xx^{(n)};\Lambda(\mathcal{M})).
		\]
		The inverse system $\{K_{\bullet}(\xx^{(n)};M/M_m)\}_{m \geq 1}$ is degree-wise surjective. 
		Whence for its $0$-th homology there is a short exact sequence
		\[
		0 \to \varprojlim{}^1_m H_1(\xx^{(n)};M/M_m) \to H_0(\xx^{(n)};\Lambda(\mathcal{M})) 
		\to \varprojlim{}_m H_0(\xx^{(n)};M/M_m) \to 0
		\]
		(see \cite[1.2.8]{SS}). It forms an exact sequence of inverse systems on $n$. 
		By passing to the inverse limit it 
		provides the short exact sequence of the statement 
		since $\varprojlim{}^1_n \varprojlim{}^1_m H_1(\xx^{(n)};M/ M_m) =0$ because of the underlying 
		bi-countable indexed system (see  the spectral 
		sequence in \cite{Rje}). Whence the statement follows.
	\end{proof}
	
	The previous result is an extension of  \cite[Lemma 1.6]{GM} to the case of a sequence of elements 
	and a more general filtration. Namely, 
	it was shown by Greenlees and May  that the vanishing of 
	$\varprojlim{}_n \varprojlim{}^1_m H_1(x^n;M/\mathcal{I}^m M)$ implies the isomorphism  
	$\Lambda^x(\Lambda^{\mathcal{I}}(M)) \cong \Lambda^{(x,\mathcal{I})}(M)$. By \ref{gm+} 
	the vanishing is also necessary for the isomorphism.

	For any set $S$ we define also $\Lambda(\mathcal{M}^{(S)}) = \varprojlim M^{(S)}/M_n^{(S)} \cong 
	\varprojlim ((M/M_n)^{(S)})$. 
	For an element $x \in R$ we put - as before - 
	$$
	\Lambda((\mathcal{M}/x\mathcal{M})^{(S)})= \varprojlim M^{(S)}/(xM,M_n)^{(S)} \cong 
	\varprojlim ((M/(xM,M_n)^{(S)})).
	$$ 
	Moreover, we study when the inverse system $\{M_n:_M x^n/M_n\}_{n \geq 1}$ with the multiplication by $x$ is pro-zero. That is, when for each $n \geq 1$ there is an $m \geq n$ such that the multiplication map 
	$$
	M_m:_M x^m/M_m  \stackrel{x^{m-n}}{\longrightarrow} M_n :_M x^n/M_n
	$$ 
	is zero. This is equivalent to the 
	inverse system $\{H_1(x^n; M/M_n)\}_{n \geq 1}$ being pro-zero, where  $H_1(x^n; M/M_n)$ 
	denotes the Koszul homology of $M/M_n$ with respect to the element $x^n$. In other words, 
	for each integer $n \geq 1$ there is an $m \geq n$ such that $M_m:_M x^m 
	\subseteq M_n:_M x^{m-n}$. Note that, if $M_n =:N$ for all $n \geq 1$, then 
	$\{H_1(x^n; M/N)\}_{n \geq 1}$ is pro-zero if and only if $M/N$ is of bounded $x$-torsion. With this in mind 
	we shall continue with an extension of \ref{apa-11}.
	
	\begin{theorem} \label{apa-0}
		With the previous notation the following conditions are equivalent:
		\begin{itemize}
			\item[(i)] The inverse system $\{H_1(x^n;M/M_n)\}_{n \geq 1}$ is pro-zero.
			\item[(ii)] $\check{H}^x_1(\Lambda(\mathcal{M}^{(S)})) = 0$ and $\check{H}^x_0(\Lambda(\mathcal{M}^{(S)})) \cong 
			\Lambda((\mathcal{M}/x\mathcal{M})^{(S)})$ for any set $S$.
			\item[(iii)] Condition (ii) holds for $S$ a set of a single element and $S = \mathbb{N}$. 
			\item[(iv)] $\varprojlim H_1(x^n;Y_n) = \varprojlim^1 H_1(x^n;Y_n) = 0$ for both $Y_n = M/M_n$ and  $Y_n= M/M_n[T]$.
			\item[(v)] $\varprojlim H_1(x^n;M/M_n[T] )= \varprojlim^1 H_1(x^n;M/M_n[T]) = 0$. 
		\end{itemize}
	\end{theorem}
	
	\begin{proof}
		(i) $ \Longrightarrow$ (ii): We put $X = M^{(S)}$ and $X_n = (M_n)^{(S)}$. Then it follows that  
		$\{H_1(x^n;X/X_n)\}_{n \geq 1}$ is pro-zero too since the Koszul homology commutes 
		with direct sums, therefore 
		\[
		\varprojlim H_1(x^n;X/X_n) = \varprojlim{}^1 H_1(x^n;X/X_n) = 0.
		\] 
		Furthermore there are isomorphisms 
		$$
		\varprojlim_m H_1(x^n; X/X_m) \cong 
		\varprojlim_m \Hom_R(R/x^nR,X/X_m) \cong 
		H_1(x^n; \Lambda(X)) 
		$$ 
		for all  $n \geq 1$.  
		We have the bi-indexed system 
		$\{H_1(x^n;X/X_m)\}_{n \geq 1,m \geq 1}$  and the diagonal system $\{H_1(x^n;X/X_n)\}_{n \geq 1}$ 
		cofinal in it. 
		There are the isomorphisms and the vanishing 
		\[
		\varprojlim_n H_1(x^n;\Lambda(X))  \cong 
		\varprojlim{}_n \varprojlim{}_m H_1(x^n;X/X_m) \cong \varprojlim_{n,m} H_1(x^n;X/X_m) = 0. 
		\]
		By virtue of Roos' spectral sequence  (see \cite{Rje} or \cite[5.8.7]{Wc}) there is a short exact 
		sequence 
		\[
		0 \to \varprojlim{}^1_n \varprojlim{}_m H_1(x^n;X/X_m) \to \varprojlim{}^1_{n,m}
		H_1(x^n;X/X_m) \to  \varprojlim{}_n \varprojlim{}^1_m 
		H_1(x^n;X/X_m) \to 0 \tag{$\#$}
		\] 
		and a similar one with $m,n$ reversed. 
		This implies  
		the vanishing $\varprojlim{}^1 H_1(x^n;\Lambda(X)) = 0$ and also  $\varprojlim{}_n \varprojlim{}^1_m H_1(x^n;X/X_m) = 0$. 
		By view of \ref{apa-2} and \ref{gm+} this proves the claim. \\ 
		(ii) $\Longrightarrow$ (iii): This holds trivially. \\
		(iii) $\Longrightarrow$ (iv): 
		By 	\ref{apa-2} the assumption implies that 
		$ \varprojlim  H_1(x^n; \Lambda(\mathcal{X})) = \varprojlim^1   H_1(x^n;\Lambda(\mathcal{X})) =0$ 
		for $\mathcal{X} =\mathcal{M} $ and $\mathcal{M}[T]$.  Put $X/X_m = \mathcal{M}_m$. 
		Because 
		$\Lambda(\mathcal{X}) \cong \varprojlim_m X/X_m$ and since the inverse limit  commutes 
		(as above)  with the first Koszul homology it follows that 
		\[
		\varprojlim_n  \varprojlim_m H_1(x^n; X/X_m ) = 
		\varprojlim_n{}^1  \varprojlim_m H_1(x^n; X/X_m) = 0 \tag{$\star$}.
		\]
		The first vanishing implies that 
		$\varprojlim H_1(x^n; M/M_n) = 0$.  
		In order to continue note that the isomorphism of the assumption $\check{H}^x_0(\Lambda(\mathcal{X})) 
		\cong \Lambda(\mathcal{X}/x\mathcal{X})$
		factors through 
		\[
		\check{H}^x_0(\Lambda(\mathcal{X})) \stackrel{\beta}{\longrightarrow}
		\Lambda^x(\Lambda(\mathcal{X}))  \stackrel{\gamma}{\longrightarrow}
		\Lambda(\mathcal{X}/x\mathcal{X})
		\]
		surjections $\beta$ (see \ref{apa-2}) and $\gamma$ (see \ref{gm+}).
		Whence  $ \Lambda^x(\Lambda(\mathcal{X}))  \to \Lambda(\mathcal{X}/x\mathcal{X}) $
		is an isomorphism 
		and $\varprojlim_n \varprojlim{}^1_m H_1(\xx^{(n)};M/M_m) = 0$ (see \ref{gm+}). 
		Therefore 
		\[
		\varprojlim{}_n^1\varprojlim{}_m H_1(x^n;X/X_m) = \varprojlim{}_n 
		\varprojlim{}_m^1 H_1(x^n;X/X_m) = 0.
		\]
		By Roos' exact sequence above (see $(\#)$) $\varprojlim^1  H_1(x^n; M/M_n) = 0$, as required.\\
		(iv) $\Longrightarrow$ (v): This is obvious.\\
		(v) $\Longrightarrow$ (i): 
		The Koszul homology commutes with direct sums. Therefore 
		the implication follows by virtue of \ref{apn-3}.
	\end{proof}
	
	The implication (i) $\Longrightarrow$ (ii) in \ref{apa-0} is a generalization of 
	\cite[Proposition 1.7]{GM}. 
	Furthermore, a certain generalization of bounded torsion to the study of sequences  was invented by Greenlees and May (see \cite{GM}) and Lipman et al. (see \cite{ALL1}), namely:
	
	\begin{definition} \label{apa-6}
		(A) Let $\xx = x_1,\ldots,x_r$ denote a sequence of elments of $R$. For an $R$-module $M$ it is 
		called $M$-pro-regular if the inverse systems with the multiplication map by $x_i^n$
		\[
		\{(x_1^n,\ldots,x_{i-1}^n)M :_M x_i^n/(x_1^n,\ldots,x_{i-1}^n)M \}_{n \geq 1}, \quad i = 1,\ldots,r,
		\]
		are pro-zero. This  is equivalent to saying that the inverse systems 
		$\{H_1(x_i^{(n)};M/\xx_{i-1}^{(n)}M)\}_{n \geq 1}$ are pro-zero for $i = 1, \ldots, r$. 
		For a sequence of elements $\xx = x_1,\ldots,x_r$ we specify the 
		subsystems $\xx_i = x_1,\ldots,x_i$ for $i = 0,\ldots, r-1$.\\
		(B)  The notion of pro-zero is equivalent to say that for $i = 1,\ldots,r$ and 
		any positive integer $n$ there is an integer 
		$m \geq n$ such that 
		\[
		(x_1^m,\ldots,x_{i-1}^m)M :_M x_i^m \subseteq 	(x_1^n,\ldots,x_{i-1}^n)M  :_M x_i^{m-n}.  
		\]
		Note that an element $x\in R$ is $M$-pro-regular if and only if $M$ is of bounded $x$-torsion.
	\end{definition}
	
	For a discussion of the notions of pro-regularity of Greenlees and May (see \cite{GM}) resp. Lipman (see \cite{ALL1}) 
	we refer to \cite{Sp12}. Moreover, it follows that an $M$-pro-regular sequence  is also 
	$M$-weakly pro-regular (see e.g. \cite[Theorem 2.4]{Sp12}), while the converse does not hold 
	(see \cite{Lj1}).  For a homological characterization of $M$-pro-regular sequences in terms 
	of injective modules we refer to \cite[Theorem 2.1]{Sp12}. Here we add a slight extension of 
	\cite[Theorem 2.1]{Sp12}.  
	
	\begin{theorem} \label{apa-7}
		Let $\xx = x_1,\ldots,x_r$ denote an ordered sequence of elements of $R$. Let $M$ denote an $R$-module. 
		Then the following conditions are equivalent.
		\begin{itemize}
			\item[(i)] The sequence $\xx$ is $M$-pro-regular. 
			\item[(ii)] The sequence $\xx$ is $(M\otimes_RF)$-pro-regular for any flat $R$-module $F$.
			\item[(iii)] $\check{H}^1_{x_i}(\Gamma_{\xx_{i-1}}(\Hom_R(M,I) ) = 0$ for $i = 1,\ldots,k$ and any injective $R$-module $I$. 
			\item[(iv)] $\check{H}^1_{\xx_i}(\Hom_R(M,I) ) = 0$ for $i = 1,\ldots,k$ and any injective $R$-module $I$. 
		\end{itemize}
	\end{theorem}
	
	\begin{proof}
		For the equivalence of the first three conditions we refer to \cite[Theorem 2.1]{Sp12}. For the proof of 
		(iii) $\Longleftrightarrow$ (iv) we put $X = \Hom_R(M,I)$ and recall the following short 
		exact sequence 
		\[
		0 \to \check{H}^1_{x_i}(\check{H}^0_{\xx_{i-1}}(X) ) \to \check{H}^1_{\xx_i}(X) 
		\to \check{H}^0_{x_i}(\check{H}^1_{\xx_{i-1}}(X)) \to 0
		\]
		for $i = 1,\ldots,r$, (see \cite[6.1.11]{SS} or \cite[8.1 (b)]{Sp11}). Then note that $\Gamma_{\xx_{i-1}}(X) 
		\cong \check{H}^0_{\xx_{i-1}}(X)$. If (iv) holds the claim in (iii) follows easily. For the converse 
		we have $\check{H}^1_{\xx_i}(X) \cong \check{H}^1_{x_i}(\check{H}^0_{\xx_{i-1}}(X)) = 0$ 
		for $i = 1,\ldots,r$ and inductively 
		the vanishing of $\check{H}^1_{\xx_i}(X)$ for $i = 1,\ldots,r$, recall that $\check{H}^1_{\xx_{i-1}}(X) = 0$ 
		by the inductive step. This proves (iii).  
	\end{proof}
	
	Recall that \ref{apa-7} provides a characterization of $M$-pro-regular sequences 
	in terms of \v{C}ech cohomology. 
	In the following we shall prove a characterization in terms of \v{C}ech homology.  
	This depends upon the 
	results of pro-zero inverse systems as shown above. 
	
	\begin{theorem} \label{apa-8}
		Let $\xx = x_1,\ldots,x_r$ denote a sequence of elements of $R$. For an $R$-module $M$ the following 
		conditions are equivalent: 
		\begin{itemize}
			\item[(i)] The sequence $\xx$ is $M$-pro-regular. 
			\item[(ii)] $\check{H}^{x_i}_0(\Lambda^{\xx_{i-1}}(M^{(S)}) )\cong 
			\Lambda^{\xx_i}(M^{(S)})$ and $\check{H}^{x_i}_1(\Lambda^{\xx_{i-1}}(M^{(S)})) = 0$ 
			for $i = 1,\ldots,r$ and any set $S$.
			\item[(iii)] $\check{H}^{x_i}_0(\Lambda^{\xx_{i-1}}(X)) \cong 
			\Lambda^{\xx_i}(X)$  and $\check{H}^{x_i}_1(\Lambda^{\xx_{i-1}}(X)) = 0$ for $i = 1,\ldots,r$ and $X = M,M[T]$.
			\item[(iv)] $\check{H}_0^{\xx_i}(X) \cong \Lambda^{\xx_i}(X)$ and 
			$\check{H}_1^{\xx_i}(X) = 0$ for $i = 1,\ldots,r$ and $X = M,M[T]$.
			\item[(v)] $\Lambda^{\xx_{i-1}}(X)$ is of bounded $x_i$-torsion for $i = 1,\ldots,r$ and $X = M,M[T]$.
		\end{itemize}
	\end{theorem}
	
	\begin{proof}
		First note that $\xx$ is $M^{(S)}$-pro-regular for any set $S$. It turns out since $R/\xx_i^{(n)}R$ is finitely generated and $\Hom_R(R/\xx_i^{(n)}R,  \cdot)$ commutes with direct sums.  Because of 
		\[
		\xx_{i-1}^{(n)} M^{(S)} :_ {M^{(S)}} 
		x_i^n /\xx_{i-1}^{(n)} M^{(S)}  \cong H_1(x_i^n;H_0(\xx_{i-1}^{(n)};M^{(S)}))
		\]
		for all $n \geq 0$ and $i = 1, \ldots,r$, it follows that the corresponding 
		inverse systems are isomorphic and pro-zero. Note that 
		$H_0(\xx_{i-1}^{(n)};M^{(S)}) \cong M^{(S)}/\xx_{i-1}^{(n)} M^{(S)}$. 
		Moreover the condition 
		and 
		Theorem \ref{apa-0} proves the equivalence of the first three statements.\\
		(iii) $\Longleftrightarrow $(iv) :  By view of  \cite[8.1]{Sp11} there are short exact sequences 
		\[
		0 \to \check{H}^{x_i}_0(\check{H}^{\xx_{i-1}}_j(X )) \to \check{H}^{\xx_i}_j(X) \to 
		\check{H}^{x_i}_1(\check{H}^{\xx_{i-1}}_{j-1}(X) ) \to 0 	\tag{$\dagger$}
		\]
		for $i = 1,\ldots,r$ and $j = 0,1$.  Then the equivalence is easily seen by the exact sequences. More precisely, 
		(iii) $\Longrightarrow$ (iv) 
		follows by increasing induction on $i$ starting at $i =1$. The converse follows similarly. 	\\  
		(v) $\Longrightarrow$ (iii): The assumption in (v) implies the vanishing 
		\[
		\varprojlim H_1(x_i^n ;\Lambda^{\xx_{i-1}}(X)) = \varprojlim{}^1 H_1(x_i^n ; \Lambda^{\xx_{i-1}}(X)) = 0.
		\] 
		By virtue of \ref{apa-2} it follows that $\check{H}_1^{\xx_i}(\Lambda^{\xx_{i-1}}(X)) = 0$ and 
		$\check{H}_0^{\xx_i}(\Lambda^{\xx_{i-1}}(X)) \cong \varprojlim H_0(x_i^n; \Lambda^{\xx_{i-1}}(X))$. 
		Now  we have $\varprojlim H_0(x_i^n; \Lambda^{\xx_{i-1}}(X)) \cong \varprojlim_n \varprojlim_m 
		X/(x_i^n,\xx_{i-1}^{(m)})X \cong \Lambda^{\xx_i}(X)$, which proves the claim in (iii).\\
		(iii) $\Longrightarrow $(v): The statement yields $\varprojlim H_1(x_i^n;\Lambda^{\underline{x}_{i-1}}(X)) = 
		\varprojlim{}^1H_1(x_i^n;\Lambda^{\underline{x}_{i-1}}(X))= 0$. For a fixed $n$ and $j = 0,1$ 
		we have the short exact sequences
		\[
		0 \to \varprojlim{}^1_m H_{j+1}(x_i^n; X/\xx_{i-1}^{(m)}X) \to H_j(x_i^n;\Lambda^{\underline{x}_{i-1}}(X)) 
		\to  \varprojlim{}_m H_j(x_i^n; X/\xx_{i-1}^{(m)}X) \to 0.
		\]
		This follows since the inverse system for  $\varprojlim{}_m K_{\bullet}( x_i^n;X/\xx_{i-1}^{(m)}X) 
		\cong K_{\bullet}(x_i^n;\Lambda^{\underline{x}_{i-1}}(X) )$ has degree wise surjective maps. For $j = 1$ 
		it yields that 
		\[
		0 = \varprojlim{}_n H_1(x_i^n;\Lambda^{\underline{x}_{i-1}}(X)) \cong 
		\varprojlim{}_n  \varprojlim{}_m H_1(x_i^n; X/\xx_{i-1}^{(m)}X) \cong 
		\varprojlim{}_n H_1(x_i^n; X/\xx_{i-1}^{(n)}X).
		\]
		It remains to show the vanishing of $ \varprojlim{}^1_n H_1(x_i^n;  X/\xx_{i-1}^{(n)}X)$. First note that 
		the above short exact sequence for $j = 1$ provides that 
		$\varprojlim^1_n  \varprojlim{}_m H_1(x_i^n; X/\xx_{i-1}^{(m)}X) = 0$. The same sequence for $j=0$ yields that 
		$\varprojlim_n  \varprojlim{}^1_m H_1(x_i^n; X/\xx_{i-1}^{(m)}X) = 0$. Then the above sequence $(\#)$ (see proof of \ref{apa-0})
		with $m,n$ reversed proves the vanishing $ \varprojlim{}^1_n H_1(x_i^n; X/\xx_{i-1}^{(n)}X) =0$.
	\end{proof}
	
	\section{A global variation} 
	As before, let $R$ denote a commutative ring. For an element $f \in R$ we write $D(f) = \Spec R \setminus V(f)$. Note that 
	$D(f)$ is an open set in the Zariski topology of $\Spec R$. For $f 
	\in R$ there is a natural map $\Spec R_f \to \Spec R$ that induces a homeomorphism between $\Spec R_f$ 
	and $D(f)$. Since $\Spec R = \cup_{f \in R} D(f)$ and since $\Spec R$ is quasi-compact 
	there are finitely many $f_1, \ldots, f_r \in R$ such that $\Spec R = \cup_{i=1}^r D(f_i)$. 
	Then we recall the following definitions (see \cite{Sp12}).
	
	\begin{definition} \label{apa-9}
		(A) A sequence $\ff = f_1,\ldots,f_r$ of elements of $R$ is called a covering sequence if $\Spec R = 
		\cup_{i=1}^r D(f_i)$. This is equivalent to saying that $R = \ff R$. Moreover, if $\ff$ is a covering sequence 
		then the natural map $M \to \oplus_{i=1}^r M_{f_i}$ is injective for any $R$-module $M$  as easily seen. \\
		(B) An ideal $\mathcal{I} \subset R$ is 
		called an effective Cartier divisor 
		if there is a covering sequence $\ff = f_1,\ldots,f_r$ such that $\mathcal{I} R_{f_i}  = x_i R_{f_i}, i = 1,\ldots,r,$ 
		for non-zerodivisors $x_i/1$ of $R_{f_i}$ with $x_i \in R$. It follows that $\mathcal{I} \subseteq (x_1,\ldots,x_r)R$. \\
		(C) Let $\mathcal{I}$ denote an effective Cartier divisor and $x \in R $. The pair $(\mathcal{I}, x)$ is 
		called pro-regular if 
		for any integer $n$ there is an integer $m\geq n$ such that $\mathcal{I}^m: x^m  \subseteq 
		\mathcal{I}^n : x^{m-n}$ . This is  consistent with the definition in \cite{GM} (see \ref{apa-6}) and is equivalent 
		to the fact that for each $n$ there is an integer $m \geq n$ such that the multiplication map 
		$
		\mathcal{I}^m :_R x^m /\mathcal{I}^m \stackrel{x^{m-n}}{\longrightarrow} 
		\mathcal{I}^n :_R x^n /\mathcal{I}^n
		$
		is the zero map. Moreover, the pair  $(\mathcal{I},x)$ is  pro-regular if and only if the inverse system 
		$\{H_1(x^n;R/\mathcal{I}^n)\}_{n \geq 1}$ is pro-zero.
	\end{definition}
	
	For the following we need a technical result about Cartier divisors and their relation to pro-regularity.

	\begin{lemma} \label{apa-y}
		Let $\mathcal{I} \subseteq R$ be an effective Cartier divisor with the covering sequence 
		$\ff = f_1,\ldots, f_r$ such that $\mathcal{I} R_{f_i}  = x_i R_{f_i}, i = 1,\ldots,r,$ for 
		non-zerodivisors $x_i/1$ of $R_{f_i}$.  For an element $x \in R$ the following conditions are equivalent:
		\begin{itemize}
			\item[(i)]  $R/\mathcal{I}$ is of bounded $x$-torsion. 
			\item[(ii)] $R_{f_i}/x_i R_{f_i}$ is of bounded $x/1$-torsion for $i = 1,\ldots,r$.
			\item[(iii)] $x_i/1, x/1$ is pro-regular in $R_{f_i}$ for $i = 1,\ldots,r$ in the sense of \ref{apa-6}. 		
			\item[(iv)] $(\mathcal{I},x)$ is pro-regular in the sense of \ref{apa-9}.
		\end{itemize}
	\end{lemma}
	
	\begin{proof}
		(i) $\Longleftrightarrow$ (ii): For each pair of integers $m \geq n\geq 1$ we have the 
		following commutative diagram where the horizontal maps are injections
		\[
		\begin{array}{ccc}
			\mathcal{I} :_R x^m/\mathcal{I} & \to & \oplus_{j=1}^r  (x_i  R_{f_i} :_{R_{f_i} }  x^m/1)/x_i R_{f_i}\\
			\downarrow^{x^{m-n}}& & \downarrow ^{\oplus (x^{m-n}/1)}\\
			\mathcal{I} :_R x^n/\mathcal{I}& \to & \oplus_{j=1}^r  (x_i  R_{f_i} :_{R_{f_i} }  x^n/1)/x_i R_{f_i}
		\end{array}
		\]
		which proves the equivalence. \\
		(ii) $\Longleftrightarrow$ (iii): Note that $x_i/1, x/1$ is pro-regular if and and only if $R_{f_i}/x_i^k R_{f_i}$ 
		is of bounded $x/1$-torsion for all $k \geq 1$. The equivalence follows easily: First note that 
		$x_i/1$ is $R_{f_i}$-regular. Then use induction on the short exact sequence 
		$$
		0 \to x_i^kR_{f_i}/x_i^{k+1}R_{f_i} \to R_{f_i}/x_i^{k+1}R_{f_i} \to R_{f_i}/x_i^kR_{f_i} \to 0
		$$
		and recall that  $ x_i^kR_{f_i}/x_i^{k+1}R_{f_i} \cong  R_{f_i}/x_iR_{f_i}$.\\
		(iii) $\Longleftrightarrow$ (iv): 
		The equivalence comes out by the following modification  of the above commutative diagram 
		\[
		\begin{array}{ccc}
			\mathcal{I}^m :_R x^m/\mathcal{I}^m& \to & \oplus_{j=1}^r  (x_i ^mR_{f_i} :_{R_{f_i} }  x^m/1)/x_i^mR_{f_i}\\
			\downarrow^{x^{m-n}}& & \downarrow ^{\oplus (x^{m-n}/1)}\\
			\mathcal{I}^n :_R x^n/\mathcal{I}^n& \to & \oplus_{j=1}^r  (x_i ^nR_{f_i} :_{R_{f_i} }  x^n/1)/x_i^nR_{f_i}.
		\end{array}
		\]
		Recall that the horizontal maps are injective (see also \cite{Sp12}). 
	\end{proof}
	
	Next we apply the previous investigations to the case when the pair $(\mathcal{I},x)$ is 
	pro-regular in the sense of \ref{apa-9}.
	
	\begin{lemma} \label{apa-10}
		Let $\mathcal{I} \subseteq R$ be an effective Cartier divisor with the covering sequence $\ff = f_1,\ldots, f_r$ 
		such that $\mathcal{I} R_{f_i}  = x_i R_{f_i}, i = 1,\ldots,r,$ 
		for non-zerodivisors $x_i/1$ of $R_{f_i}$. For an element $x \in R$ the following conditions are equivalent:
		\begin{itemize}
			\item[(i)] $R/\mathcal{I}$ is of bounded $x$-torsion.
			\item[(ii)] $\check{H}^x_1( (R/\mathcal{I})[T]) = 0$ and $\check{H}^x_0((R/\mathcal{I})[T]) 
			\cong \Lambda^x((R/\mathcal{I})[T])$.
			\item[(iii)] $\check{H}_1^x(\Lambda^{\mathcal{I}}(X))= 0$ and $\check{H}_0^x(\Lambda^{\mathcal{I}}(X)) \cong 
			\Lambda^{(x,\mathcal{I})}(X)$ for $X = R, R[T]$.
			\item[(iv)] $\Lambda^{\mathcal{I}}(R)$ and $\Lambda^{\mathcal{I}}(R[T])$ are of bounded $x$-torsion.
		\end{itemize}
	\end{lemma}
	
	\begin{proof}
		First note that by \ref{apa-y} $\{H_1(x^n;R/\mathcal{I})\}_{n \geq 1}$ is pro-zero 
		if and only if $\{H_1(x^k;R/\mathcal{I}^k)\}_{k \geq 1}$ is pro-zero. 
		Then the equivalence of (i) and (ii) follows by \ref{apa-4}. 
		Moreover, by \ref{apa-0} the pro-zero property of the second inverse system  above 
		implies the equivalence to (iii). 
		Finally the equivalence of (iii) and (iv) is a consequence of \ref{apa-8} and \ref{gm+} since 
		$\varprojlim{}_n \varprojlim{}^1_m H_1(x^n;R/\mathcal{I}^m) = 0$. 
	\end{proof}
	
	In the following we shall give a comment of the previous investigations to the recent work 
	of Bhatt and Scholze (see \cite{BhS}) completing the results of 
	\cite{Sp12}. To this end let $p \in \mathbb{N}$ denote a prime number 
	and let $\mathbb{Z}_p := \mathbb{Z}_{\mathfrak{p}}$ the localization at  the prime ideal 
	$(p)= \mathfrak{p}\in \Spec \mathbb{Z}$. In the following let $R$ be a $\mathbb{Z}_p$-algebra. 
	
	\begin{definition} \label{def-5} (see \cite[Definition 1.1]{BhS})
		A prism is a pair $(R,\mathcal{I}$) consisting of a $\delta$-ring $R$ (see \cite[Remark 1.2]{BhS}) and 
		a Cartier divisor $\mathcal{I}$ on $R$ satisfying the following two conditions.
		\begin{itemize}
			\item[(a)] The ring $R$ is $(p,\mathcal{I})$-adic complete.
			\item[(b)] $p \in \mathcal{I} + \phi_R(\mathcal{I})R$, where $\phi_R$ is the lift of the Frobenius on $R$ 
			induced by its $\delta$-structure (see \cite[Remark 1.2]{BhS}).
		\end{itemize}
	\end{definition}
	
	With the previous definition there is the following application of our results.
	
	\begin{corollary} \label{cor-6}
		Let $(R, \mathcal{I})$ denote a prism. Then the following conditions are equivalent:
		\begin{itemize}
			\item[(i)] $\mathcal{I}$ is of bounded $p$-torsion.
			\item[(ii)] The pair $(\mathcal{I}, p)$ is pro-regular in the sense of \ref{apa-9}. 
			\item[(iii)]  $\check{H}^1_x(\Hom_R(R/\mathcal{I},I)) =0$  for any injective $R$-module $I$.
			\item[(iv)] $\check{H}_0^{pR}(\Lambda^{\mathcal{I}}(R^{(S)}) \cong \Lambda^{(pR,{\mathcal{I}})}(R^{(S)}))$ 
			and $\check{H}_1^{pR}(\Lambda^{\mathcal{I}}(R^{(S)}) = 0$ for any set $S$.
			\item [(v)] $\Lambda^{\mathcal{I}}(R^{(S)})$ and $\Lambda^{\mathcal{I}}(R^{(S)})$ 
			are of bounded $p$-torsion  for any set $S$.. 
			\item[(vi)] $\Lambda^{\mathcal{I}}(R)$ and $\Lambda^{\mathcal{I}}(R[T])$ are of bounded $p$-torsion. 
		\end{itemize}
	\end{corollary}
	
	\begin{proof}
		This is a consequence of \ref{apa-y}, \ref{apa-10} and \ref{apa-7}.
	\end{proof}
	
	Note that \ref{cor-6} is an essential improvement of \cite[Corollary 4.5]{Sp12}, where it was shown 
	that (i) implies the equivalent conditions (ii) and (iii). 
	
\medskip

	{\bfseries{Acknowledgement.}} Many thanks to the reviewer for the careful reading of the manuscript 
	and the suggestions for improving the text and correcting references.

	\bibliographystyle{siam}

	\bibliography{hart}
	
\end{document}